\documentclass[12pt]{amsart}
\usepackage{graphicx}
\usepackage[english]{babel}
\usepackage{blindtext}
\usepackage[utf8]{inputenc}
\usepackage{fancyhdr}
\usepackage{amsmath}
\usepackage{amssymb}
\usepackage{mathtools}
\usepackage{tikz-cd}
\usepackage{enumitem}
\usepackage{latexsym,amsfonts,amssymb,amsthm,amsmath}
\usepackage{amssymb}

\def\A{\mathbb A}
\def\N{\mathbb N}

\def\Q{\mathbb Q}
\def\P{\mathbb P}
\def\O{\mathcal O}

\def\Z{\mathbb Z}
\def\L{\mathcal L}

\newtheorem{thm}{Theorem}[section]
\newtheorem{lema}{Lemma}[section]
\newtheorem{coro}{Corollary}[section]
\newtheorem{defi}{Definition}[section]
\newtheorem{rema}{Remark}[section]
\newtheorem{prop}{Proposition}[section]
\newtheorem{exm}{Example}[section]
\begin{document}

\title{ Some evidence for the existence of Ulrich bundles}
\author{ \c Stefan Deaconu}

\begin{abstract} The question of existence of Ulrich bundles on nonsingular projective varieties is posed here in weaker terms: either to find a K--theoretic solution,  or to find one in the derived category of the variety.  We observe  that if any motivic vector bundle is algebraic,  there is always a solution in the Grothendieck group.  Also,  by considering the derived problem,  it is noted a formal way of producing Ulrich sheaves on a surface.
\end{abstract}
\maketitle
\section{Introduction} 

In general, by $k$ it is meant an algebraically closed field. 

\begin{defi} An Ulrich bundle on a projective variety $(X, \O_X(1))$ of dimension $d$ over $k$ is a locally free sheaf $\mathcal{E}$ such that
\begin{align*}
H^i(X, \mathcal{E}\otimes\mathcal{O}_X(j))=0\text{ for all }-d\leq j\leq -1\text{ and all }i\in\N.
\end{align*}
\end{defi}

There are two others equivalent definitions for an Ulrich bundle \cite[Theorem 1.1]{beu}. The guiding problem is to find such bundles on any nonsingular projective variety over $k$ (this geometric form of the problem was first raised in \cite{ES}).  In fact,  the guiding problem is the fundamental problem of the extrinsic point of view: to pin down equations of varieties by resolving sheaves.  For curves, this was accomplished by Petri \cite[A.I]{mum2}, \cite[III.3]{acgh},  but in general less is known. Now, the idea is that in the presence of an Ulrich bundle, the Chow form can (almost) be determined linearly.  For instance: 

\begin{exm} {\em If $X$ is a hypersurface of degree $d$ in some projective space $\P_k^N$, and $f=0$ is its equation, then the existence of an Ulrich bundle of rank $r$ on $X$ allow us to write $f^r=det(l_{ij})$, where $(l_{ij})$ is a $d\times d$ matrix of linear forms on $\P_k^N$ (\cite[Proposition 1]{beu}, or \cite[Section 3]{ES} for more)}.
\end{exm}

We note that in the case of curves, one can use invariant polynomials in the Chow form of the curve in order to embed the moduli space of curves in a projective space \cite[A.II]{mum2}. Thus, if any nonsingular projective variety is to carry an Ulrich bundle,  one might hope that this will improve our understanding of moduli in larger dimensions.

The purpose of this note is to give an evidence to this question in the form of a weak (K--theoretic) solution (Corollary \ref{sol}).  More precisely,  the idea is this: one first notice that, over $k=\mathbb{C}$,  a topological solution to the problem can be produced easy enough,  so then the question is to find a more algebraic language allowing to use homotopy reasons; such a language was introduced by Morel and Voevodsky,  which was used to establish a notion of {\em motivic vector bundle} as presented in \cite{af}.  One can always pass from a vector bundle to a motivic one,  and under the assumption that this process produce all motivic bundles (only in the smooth context),  a K--theoretic Ulrich bundle can be found.  

In the final part,  the notion of an Ulrich sheaf is introduced in derived version,  and it is noticed that a derived Ulrich sheaf on a surface always gives an Ulrich bundle (Corollary \ref{manuf}).

\section{Preliminaries}

By $K_{0, \mathbb{Q}}$ we shall denote the functor $X\mapsto K_0(X)\otimes_{\mathbb{Z}}\mathbb{Q}$, associating to a noetherian scheme $X$ the Grothendieck group of the category of coherent sheaves on $X$, extended over the rationals. Also, $K^0_{\Q}$ will denote the functor given by $X\mapsto K^0(X)\otimes\Q$, where $K^0(X)$ is the Grothendieck group of the category of locally free sheaves on $X$ (see \cite{man}). Recall the following fundamental fact: 

\begin{thm}\label{syz}(Kleiman,   \cite[1.9]{man}) If $X$ is a regular noetherian scheme having an ample invertible sheaf on it, the tautological map
\begin{align*}
K^0(X)\rightarrow K_0(X)
\end{align*}
is an isomorphism of $K^0(X)$--modules.
\end{thm}

\begin{defi} For a scheme $X$, by a Jouanolou device it is meant an affine vector bundle torsor $p:\tilde{X}\rightarrow X$.
\end{defi}

In order to present the existence result concerning the notion of Jouanolou device in full generality, recall that a quasi-compact and quasi-separated scheme $X$ is called $divisorial$ if it has an ample family of invertible sheaves $\{\L_i\}_i$, i.e. the distinguished open subsets $D(f)=\{f\neq 0\}, f\in\text{H}^0(X, \L_i^{\otimes n})$, form a base for the Zariski topology (this class of varieties were introduced first in S.G.A 6). This said, we have:

\begin{lema}\label{jou} (Jouanolou-Thomason, \cite[4.4]{wei}) Any divisorial quasiseparated and quasicompact scheme admits a Jouanolou device.
\end{lema}

This trick was brought in Algebraic Geometry by Jouanolou from a construction in Algebraic Topology (as with many others ideas) in order to extend the $K$--spectrum from commutative rings to schemes (according to \cite{wei}). 

\begin{exm}\label{st} {\em For instance, the standard Jouanolou device for a projective variety over $k$ is constructed as follows. First, one observe that it suffices to construct it on a projective space (by pulling back through a finite morphism,  according to Lemma \ref{finite} below). For $X=\P_k^n$, take $\widetilde{X}$ to be the affine variety given by the complement of the incidence divisor in $\P^n_k\times(\P^n_k)^{\vee}$, together with the projection onto the first factor. This is affine indeed: it is the complement of the hypersurface given by the equation $x_0y_0+\ldots+x_ny_n=0$, which we see via the Segre embedding $\P^n_k\times\P^n_k\rightarrow\P^{n^2+2n}_k$ that this hypersurface is the preimage of a hyperplane \cite[I.Ex.2.14]{har}; moreover, each fiber is given by the hyperplanes not containing a given point, so it is isomorphic to $\A_k^n$. 

For another variant: take $\widetilde{\P_k^n}$ to be the variety given by all $(n+1)\times (n+1)$ matrices which are idempotent and have rank one. Using the characteristic polynomial, we see that this is affine, and seeing each matrix as a linear transformation, we fiber it over $\P_k^n$ by sending a matrix to its image (makes sense, since these matrices have rank $1$).  }
\end{exm}

\begin{thm}\label{desc} \cite[5.1.2]{af} Assume $\text{dim}(X)\leq 3$. If $p:\tilde{X}\rightarrow X$ is a Jouanolou device, then any locally free sheaf on $\tilde{X}$ admits a descent datum relative to $p$.
\end{thm}

\begin{rema} {\em It is hoped that Theorem \ref{desc} can be extended to any dimension \cite[5.1.1]{af}. Just to briefly mention why is this important, given a nonsingular projective variety $X$ over $k$, set $\mathcal{V}^r_{\text{alg}}(X)$ to be the collection of locally free sheaves on $X_{\text{Zar}}$, of rank $r$, and $\mathcal{V}^r_{\text{mot}}(X)$ to be that of motivic locally free sheaves of rank $r$ on $X$, which is the collection of $\A^1$--homotopy classes of maps from $X_{\text{an}}$ to the corresponding infinite grassmanian \cite[3.6.1]{af}. Also, let $\mathcal{V}^r_{\text{alg}}(\tilde{X})$ and $\mathcal{V}^r_{\text{mot}}(\tilde{X})$ be defined similarly for $\tilde{X}$, where $p:\tilde{X}\rightarrow X$ is a Jouanolou device for $X$ (see Lemma \ref{jou}). Consider the following diagram:
\[ 
\begin{tikzcd}
\mathcal{V}^r_{\text{alg}}(X) \arrow{r}{p_X} \arrow[swap]{d}{p^*} & \mathcal{V}^r_{\text{mot}}(X) \arrow{d}{p^*_{\text{mot}}} \\
\mathcal{V}^r_{\text{alg}}(\tilde{X}) \arrow{r}{p_{\tilde{X}}} & \mathcal{V}^r_{\text{mot}}(\tilde{X})
\end{tikzcd}
\]
Since $\tilde{X}$ is affine, the Morel--Schlichting theorem asserts that $\mathcal{V}^r_{\text{alg}}(\tilde{X})\simeq\mathcal{V}^r_{\text{mot}}(\tilde{X})$ functorially \cite[3.4.3]{af}. Also, $\mathcal{V}^r_{\text{mot}}(X)=\mathcal{V}^r_{\text{mot}}(\tilde{X})$ (by definition of the motivic vector bundles, since $\tilde{X}$ is an affine torsor over $X$).  

Therefore, the meaning of the positive answer to \cite[5.1.1]{af} would be that any motivic vector bundle on $X$ is algebraic.}
\end{rema}

The $K$--groups are homotopy invariant: \\

\begin{thm}\label{inv}(Grothendieck, \cite[1.2]{wei}) Let $X$ be a regular projective scheme over $k$ (not necessarily algebraically closed), and $p:\tilde{X}\rightarrow X$ a Jouanolou device. Then $p^*:K_0(X)\rightarrow K_0(\tilde{X})$ is an isomorphism
\end{thm}

\begin{rema} {\em It is also to be observed that Theorem \ref{inv} is actually using the higher $K$--groups. Indeed, in order to establish the isomorphism $K_0(X)\simeq K_0(\tilde{X})$ one has to extend the $\A^1$--homotopy invariance from commutative rings to varieties, and this is done via a Mayer--Vietoris argument, which in turn is not possible without the development of the full $K$--spectrum. As such, it is not clear that, whenever we have $[\mathcal{F}]$ in $K_0(\tilde{X})$, we may write $[\mathcal{F}]=p^*[\mathcal{E}]$ for some sheaf $\mathcal{E}$.}
\end{rema}

\begin{rema}\label{proj} {\em Recall that, given a noetherian integral affine scheme $X$ over $k$, a theorem of Serre says that the category of locally free sheaves on $X$ is equivalent to the category of finitely generated projective modules over $\Gamma (X, \mathcal{O}_X)$ (algebraic Swan's theorem). This well-known fact follows from \cite{har}.II.5.5 and Nakayama's lemma.}
\end{rema}

We end this section with two more definitions. 

\begin{defi}\label{pseudo} Let $X$ be a noetherian scheme and $\mathcal{F}, \mathcal{G}\in Coh(X)$. We say that $\mathcal{F}$ and $\mathcal{G}$ are pseudo--isomorphic if there exists a morphism of coherent sheaves $\varphi:\mathcal{F}\rightarrow\mathcal{G}$ such that $dim(supp(Ker(\varphi)))$ and $dim(supp(Coker(\varphi)))$ are strictly smaller than $dim(X)$. 
\end{defi}

\begin{defi}\label{surj} Given a covariant functor $F:\textbf{Sch}/k\rightarrow\textbf{C}$, a proper morphism of noetherian schemes $f:X\rightarrow Y$ will be called $F$--$surjective$ if the induced transfer map
\begin{align*}
Ff:F(X)\rightarrow F(Y)
\end{align*}
is surjective. 
\end{defi}

\section{ The results}

\subsection{The K--theoretic approach}

We begin with the following remark: \\

\begin{lema}\label{pseudo} Let $(X, \mathcal{O}_X(1))$ be a projective variety and $\mathcal{F}\in\textbf{Coh}(X)$ with $dim(supp(\mathcal{F}))=dim(X)$. Then there is some $n\in\N$ such that $\mathcal{O}_X(n)^{\oplus r}$ is pseudo--isomorphic to $\mathcal{F}$, where $r=rk(\mathcal{F})$ (see Definition \ref{pseudo}).
\end{lema}

\begin{proof}  According to Serre's theorem, let $n>>0$ such that $\mathcal{F}(n)$ is generated by global sections $s_1,\ldots ,s_N\in\text{H}^0(X, \mathcal{F}(n))$ \cite[II.5.17]{har}, i.e. we have a surjection
\begin{align*}
(s_1,\ldots ,s_N):\mathcal{O}_X^{\oplus N}\rightarrow\mathcal{F}(n)\rightarrow 0
\end{align*}
But locally we can be more precise: there is a Zariski open subset $U\subset X$ such that $\mathcal{F}_{|U}$ is a free $\mathcal{O}_U$ - module. Though this fact is well-known, we need to dwelve in the details of its proof, because the very notion of rank of $\mathcal{F}$ is based on it. There is a nonempty open subset $V\subset X$ such that $\mathcal{F}$ is the cokernel of some $\psi:\mathcal{L}_1\rightarrow\mathcal{L}_0$, with $\mathcal{L}_0, \mathcal{L}_1$ locally free \cite[II.Ex.5.4]{har}. Then let $U$ be the open subset where $x\rightarrow\text{rk}(\psi_x)$ is maximal (this is indeed open, thanks to semicontinuity). But then, essentially by Nakayama's lemma, it follows that (by eventualy shrinking $U$ further) $\mathcal{F}_{|U}\simeq\text{Coker}(\psi)_{|U}$ is free \cite[II.Ex.5.8]{har}; moreover, one defines $r$ to be $\text{rk}(\mathcal{F}_{|U})$.

Therefore, after possibly relabeling the sections $s_i$, by combining the above two remarks, we conclude that 
\begin{align*}
\varphi=(s_1,\ldots ,s_r):\mathcal{O}_X^{\oplus r}\rightarrow\mathcal{F}(n)
\end{align*}
is a pseudo-isomorphism (because the locus where $\varphi$ fails to be an isomorphism is a proper closed subset, namely $X-U$, and $X$ is irreducible).
\end{proof}

\begin{prop}\label{main} Any finite surjective morphism between irreducible projective varieties over $k$, $\pi:X\to Y$, is $K_{0, \Q}$ - surjective (see Definition \ref{surj}) .
\end{prop}

\begin{proof} Fix polarizations $\mathcal{O}_X(1), \mathcal{O}_Y(1)$ such that $\pi^*\mathcal{O}_Y(1)=\mathcal{O}_X(1)$, and let $[\mathcal{F}]\in K_{0, \Q}(Y)$. We then show that there is $e\in K_{0, \Q}(X)$ such that $\pi_*(e)=[\mathcal{F}]$ by induction on $\text{dim}(Z)$, where $Z=\text{supp}(\mathcal{F})$. This will be enough, because any element of the $K_0$ - group may plainly be written as a difference of the classes of two coherent sheaves.

The assertion is clearly true if $\text{dim}(Z)=0$. Indeed, any such $\mathcal{F}$ is a direct sum of skyscreapers concentrated at some points, with fiber $k$. Given a closed point $y\in Y$, let $x\in X$ be with $\pi(x)=y$; then $\pi_*\mathcal{O}_x=\mathcal{O}_y$, so $[\mathcal{O}_y]=\pi_*(\frac{1}{d}[\mathcal{O}_x])$, where $d=\text{deg}(\pi)$. So, from now on, assume $\text{dim}(Z)>0$.

First note that we may assume $Z$ to be irreducible. If not, let $Z=\bigcup_iZ_i$ be its decomposition into irreducible components. Then, denoting by $\iota_i$ the inclusion $Z_i\rightarrow Z$, the natural morphism $\varphi:\mathcal{F}\rightarrow\bigoplus_i\iota_{i*}(\mathcal{F}_{|Z_i})$ fits into
\begin{align*}
0\rightarrow\text{Ker}(\varphi)\rightarrow\mathcal{F}\overset{\varphi}{\rightarrow}\bigoplus_i\iota_{i*}(\mathcal{F}_{|Z_i})\rightarrow\text{Coker}(\varphi)\rightarrow 0,
\end{align*}
where $\text{dim(supp(Ker}(\varphi)))$ and $\text{dim(supp(Coker}(\varphi)))$ are strictly smaller than $\text{dim}(Z)$. As such, the induction hypothesis allow us to write $[\text{Ker}(\varphi)]=\pi_*(e_1)$ and $[\text{Coker}(\varphi)]=\pi_*(e_2)$, for some $e_1, e_2\in K_{0, \Q}(X)$. Now split the previous sequence in two short exact sequences using $\text{Im}(\varphi)$. If $\varphi$ is surjective, then $[\mathcal{F}]\in\text{Im}(\pi_*)$ whenever $[\mathcal{F}_{|Z_i}]\in\text{Im}(\pi_*)$; if not, then $\text{Im}(\varphi)$ falls also under the inductive hypothesis, and we have again the conclusion. 

Next, since $Z$ is assumed to be irreducible, we may use Lemma \ref{pseudo} and get a pseudo--isomorphism
\begin{align*}
0\rightarrow\mathcal{K}\rightarrow\mathcal{O}_Z(n)^{\oplus r}\rightarrow\mathcal{F}\rightarrow\mathcal{C}\rightarrow 0,
\end{align*}
were $r=\text{rank}(\mathcal{F})$ and $\text{dim(supp}(\mathcal{K}))$, $\text{dim(supp}(\mathcal{C}))$ are strictly smaller than $\text{dim}(Z)$, so that $[\mathcal{K}], [\mathcal{C}]\in\text{Im}(\pi_*)$. Therefore, as $[\mathcal{O}_Z(n)^{\oplus r}]=r[\mathcal{O}_Z(n)]$, it remains to see that $[\mathcal{O}_Z(n)]\in\text{Im}(\pi_*)$, for all $n$. 

Fix $X'$ an irreducible component of $\pi^{-1}(Z)$. Then the morphism $\pi':X'\rightarrow Z, \pi'=\pi_{|X'}$ is again finite and surjective. Apply Lemma \ref{pseudo} to $\pi_*'\mathcal{O}_{X'}$ and get an exact sequence
\begin{align*}
0\rightarrow\mathcal{K}'\rightarrow\mathcal{O}_Z(n')^{\oplus r'}\rightarrow\pi_*'\mathcal{O}_{X'}\rightarrow\mathcal{C}'\rightarrow 0,
\end{align*}
for some $n', r'\in\N$, with $\text{dim(supp}(\mathcal{K}'))$ and $\text{dim(supp}(\mathcal{C}'))$ strictly smaller than $\text{dim}(Z)$. Twisting this sequence by $\mathcal{O}_Z(n-n')$ and using the projection formula \cite[II.ex.5.1.(d)]{har} we get
\begin{align*}
0\rightarrow\mathcal{K}'(n-n')\rightarrow\mathcal{O}_Z(n)^{\oplus r'}\rightarrow\pi_*'(\mathcal{O}_{X'}(n-n'))\rightarrow\mathcal{C}'(n-n')\rightarrow 0.
\end{align*}
Note that at this point our choice of polarizations matter. Once more, by the inductive hypothesis, $[\mathcal{K}'(n-n')], [\mathcal{C}'(n-n')]\in\text{Im}(\pi_*)$, and since $\pi_*'(\mathcal{O}_{X'}(n-n'))=\pi_*(\O_{X'}(n-n'))\in\text{Im}(\pi_*)$, we conclude that $[\mathcal{O}_Z(n)]=\frac{1}{r'}[\mathcal{O}_Z(n)^{\oplus r'}]\in\text{Im}(\pi_*)$. 
\end{proof}

\begin{rema} {\em We note that the syntethic analogue of Proposition \ref{main} follows immediately. Indeed any finite surjective morphism between irreducible projective varieties $\pi:X\to Y$ is inducing a surjection $\pi_*:A(X)_{\mathbb{Q}}\to A(Y)_{\mathbb{Q}}$ at the level of Chow groups with rational coefficients, because $\pi_*([X])=d[Y]$ (where $d=deg(\pi)$), so for any $y\in A(Y)$ the projection formula gives $y=[Y]_{\cdot} y=\pi_*(\frac{1}{d}[X]_{\cdot}\pi^*y)\in A(X)_{\mathbb{Q}}$. If we assume $X$ and $Y$ to be nonsingular, using the Chern character isomorphism $K(X)_{\mathbb{Q}}\to A(X)_{\mathbb{Q}}$ \cite[9.1, 11.6]{man}, we may try to lift the surjectivity just proved to that at the $K$--theoretic level, obtaining thus another proof of Theorem \ref{main}, but the corresponding diagram is not commutative: it is the content of Grothendieck--Riemann--Roch \cite[19.5]{man}. }
\end{rema}

\begin{coro}\label{cmain} Any finite surjective morphism between nonsingular irreducible projective varieties over $k$ is $K^0_{\Q}$ - surjective.
\end{coro}

\begin{proof} This is given by Theorem \ref{syz} and Proposition \ref{main}.
\end{proof}

Also recall the following:

\begin{lema}\label{finite}(E. Noether, \cite[16]{mum}) If $(X, \O_X(1))$ is a projective variety of dimension $d$ over $k$, then there is a finite flat morphism $\pi: X\rightarrow\P_k^d$ such that $\pi^*\mathcal{O}_{\P_k^d}(1)=\mathcal{O}_X(1)$. 
\end{lema}

\begin{coro}\label{sol} Let $(X, \mathcal{O}_X(1))$ be a nonsingular irreducible projective variety over $k$. If $d=dim(X)\leq 3$, then there is some locally free sheaf $\mathcal{E}$ on $X$ and some $N\in\mathbb{N}$ such that $\pi_*([\mathcal{E}])=[\mathcal{O}_{\mathbb{P}_k^d}^{\oplus N}]$, for some finite flat morphism $\pi:X\to\mathbb{P}_k^d$.  If,  moreover,  $\pi_*\mathcal{E}$ is Gieseker semistable,  then $\mathcal{E}$ is Ulrich.  
\end{coro}

\begin{proof} Produce, with Lemma \ref{finite}, a finite flat morphism $\pi:X\rightarrow\P_k^d$ with $\pi^*\mathcal{O}_{\P_k^d}(1)\simeq\mathcal{O}_X(1)$. As such, we may use Corollary \ref{cmain} to find $e\in K^0(X)$ with $\pi_*(e)=[\mathcal{O}_{\P_k^d}^{\oplus N}]$, for some $N>>0$ (due to the open nature of flatness \cite[III.Ex.9.1]{har}, the image of $\pi$ is closed and open, with connected target, so $\pi$ is indeed surjective). Then consider a Jouanolou device $p:\tilde{X}\to X$ (Lemma \ref{jou}). Since $\text{deg}(\pi)\text{rk}(e)=N>0$, replaicing $e$ with a high enough additive multiple of it, we may assume $\text{rk}(e)>2d$. Then Remark \ref{proj} allow us to use Serre's theorem \cite[IX.4.1.(a)]{bass} in order to write $p^*e=[\mathcal{F}]$, for some locally free sheaf $\mathcal{F}$ on $\tilde{X}$. Therefore, Theorem \ref{desc} says that $\mathcal{F}\simeq p^*\mathcal{E}$, for some locally fre sheaf $\mathcal{E}$ on $X$, and hence $e=[\mathcal{E}]$ (by Theorem \ref{inv}).  The second assertion is given by \cite[Theorem 1]{mn}.
\end{proof}

\begin{rema} {\em Of course, if \cite[5.1.1]{af} is positive, then Corollary \ref{sol} is true in any dimension.  We also notice that if one would knew that $\pi_*\mathcal{E}$ is also globally generated (e.g.  if $\mathcal{E}$ is regular in the Castelnuovo--Mumford sense with respect to some polarization \cite[Lecture 14]{mum}),  then $\mathcal{E}$ would be an Ulrich bundle, according to a theorem of Van de Ven \cite[Corollary to Theorem 3.2.1]{oss}. }
\end{rema}

\subsection{The derived approach} Here we point out a formal way for producing Ulrich bundles on surfaces, but this time the choice is to work in the derived settings. The definition of an Ulrich bundle is a relative one: it is given within a projective embedding of the variety,  so in order to introduce a derived notion of this class of bundles, one must first have a notion of ampleness in a derived category. For a nonsingular projective variety $X$ over $k$ (always assumed to be like this in what follows), denote by D$_{\text{coh}}^b(X)$ the bounded derived category of the abelian category $\textbf{Coh}(X)$ of coherent sheaves on $X$ (cf.  \cite{huy}).  

\begin{defi}\label{bondal} (Bondal-Orlov,  \cite[2.72]{huy}) Let $\mathcal{A}$ be an abelian category, enriched in a monoidal category. A sequence of objects $(L_i)_{i\in \Z}$ from $\mathcal{A}$ is said to be ample if for any $A\in\text{ob}(\mathcal{A})$ there exists $i_0(A)\in \Z$ such that, whenever $i<i_0$, one has (in D$^b(\mathcal{A})$):

(a) The natural arrow Hom$(L_i, A)\otimes L_i\rightarrow A$  is an epimorphism;

(b) If $j\neq 0$, then Hom$(L_i, A[j])=0$;

(c) Hom$(A, L_i)=0$. 
\end{defi}

 The condition on enrichness is present in order to make $(a)$ unambigously. For instance, if $\mathcal{A}$ is enriched in the category of finite--dimensional vector spaces over $k$, $\textbf{Vect}_{<\infty}(k)$ (say $\mathcal{A}$ is a $k$ - linear category), given $A\in\text{ob}(\mathcal{A})$ and $V\in\text{ob}(\textbf{Vect}_{<\infty}(k))$, by $V\otimes A$ is meant the object of $\mathcal{A}$ representing the functor $\mathcal{A}\rightarrow\textbf{Vect}_{<\infty}(k), B\mapsto\text{Hom}_{\textbf{Vect}_{<\infty}(k)}(V, \text{Hom}_{\mathcal{A}}(A, B))$ (the demand of the classical adjunction). In this case, it follows immediately that $V\otimes A$ exists because dim$_kV<\infty$.

 The following says that the previous definition is good: 

\begin{exm}{\em  \cite[3.18]{huy} \label{ample}If $X$ is a projective variety over $k$ and $L$ is an ample line bundle on $X$, then $(L^{\otimes i})_{i\in \Z}$ is an ample sequence in $\textbf{Coh}(X)$. }
\end{exm}

A further abstraction of Definition \ref{bondal} steems from linear algebra reasonings: the notions of a spanning class in a triangulated category \cite[1.47, 2.73]{huy}. Anyway,  in general, we shall rely on Example \ref{ample}.

In order to state the main definition with some degree of generality, we need to have a concept of dimension. Following Cartan--Eilenberg, we proceed as follows:  

\begin{defi} (relative dimension). Let $\mathcal{D}$ be a triangulated category, equiped with a $t$ - structure $(\mathcal{D}^{\leq 0}, \mathcal{D}^{\geq 0})$. If $\mathcal{A}$ denotes the core $\mathcal{D}^{\leq 0}\cap\mathcal{D}^{\geq 0}$, then by the dimension of $(\mathcal{D}, (\mathcal{D}^{\leq 0}, \mathcal{D}^{\geq 0}))$, denoted dh$_t(\mathcal{D})$, is meant the homological dimension of $\mathcal{A}$ \cite[IV. 4.4]{gm}. 
\end{defi}

\begin{defi}\label{derul} Let $\mathcal{D}$ be a triangulated category, endowed with a $t$ - structure having the core such that dh$_t(\mathcal{D})<\infty$, and let $(L_i)_i$ be a sequence of objects in the core of the $t$ -structure. An object $E$ from $\mathcal{D}$ will be called an Ulrich object with respect to $(L_i)_i$ if Hom$(E, L_i[j])=0$ for any $j$ and any $-\text{dh}_t(\mathcal{D})\leq i\leq-1$; if, instead, these vanishings take place for each $1\leq i\leq\text{dh}_t(\mathcal{D})$, we say that $E$ is a co--Ulrich object. 
\end{defi}

\begin{rema}\label{emb} {\em Given an abelian category $\mathcal{A}$, we have an embedding $\mathcal{A}\rightarrow\text{D}^*(\mathcal{A})$, the image being the full subcategory formed by the H$^0$-complexes, for any $*\in\{-, b, +\}$ \cite[III.5.2]{gm}. }
\end{rema}

The derived category does not forget the homological algebra of the initial abelian category: 

\begin{prop}\label{notforg}\cite[2.56]{huy} Let $\mathcal{A}$ be an abelian category. If we see $\mathcal{A}$ embedded in D$(\mathcal{A})$ as in Remark \ref{emb}, then 
\begin{align*}
\text{Ext}_{\mathcal{A}}^i(A, B)\simeq\text{Ext}_{D(\mathcal{A})}^i(A, B),
\end{align*}
for all $A, B\in\text{ob}(\mathcal{A})$ and all $i$. 
\end{prop}

In what follows, we ilustrate Definition \ref{derul} in the geometric context. First, to justify the definition: \\

\begin{exm}{\em The main example is the classical notion. Let $X$ be a nonsingular projective variety over $k$ and let $d=\text{dim}X$. Say $(L_i)$ is an ample sequence in $\textbf{Coh}(X)$. An object $\mathcal{E}^{\bullet}$ from D$_{\text{coh}}^b(X)$ is to be called a derived Ulrich sheaf on $X$ with respect to $(L_i)$ if it is an Ulrich object. As dh$_{t_{can}}(\text{D}_{coh}^b(X))=\text{dh}(\textbf{Coh}(X))=\text{dim}X$, this means that Ext$_{D_{coh}^b(X)}^i(\mathcal{E}^{\bullet}, L_j^{\vee})=0$ for any $i$ and any $1\leq j\leq d$ \cite[3.7,(i)]{huy}. 

Say $\mathcal{O}_X(1)$ is a very ample line bundle which gives a projective embedding of $X$. Then $(\mathcal{O}_X(i))_{i\in \Z}$ is an ample sequence in D$_{\text{coh}}^b(X)$ (by Example \ref{ample}) and let $E$ be an Ulrich vector bundle on $X$ with respect to this embedding. If we see the associated sheaf of sections $\mathcal{E}$ in D$_{\text{coh}}^b(X)$ (by Remark \ref{emb}), then $\mathcal{E}^{\vee}$ is a derived Ulrich sheaf: for any $i$ and any $1\leq j\leq d$, we have
\begin{align*}
&0\overset{\text{\cite[Th.1.1,(2)]{beu}}}{=}\text{H}^i(X, \mathcal{E}(-j))\overset{\text{\cite[III.6.3.(c)]{har}}}{\simeq}\text{Ext}_X^i(\mathcal{O}_X, \mathcal{E}\otimes \mathcal{O}_X(-j)) \\
&\overset{\text{\cite[III.6.7]{har}}}{\simeq}\text{Ext}_X^i(\mathcal{E}^{\vee}, \mathcal{O}_X(-j))\overset{\text{Prop.\ref{notforg}}}{\simeq}\text{Ext}_{D_{coh}^b(X)}^i(\mathcal{E}^{\vee}, \mathcal{O}_X(-j)).
\end{align*}

While here, we notice how $\mathcal{E}^{\vee}$ becomes in the derived realm. There is a linear resolution $0\rightarrow\mathcal{L}_{c}\overset{d_c}{\rightarrow}\ldots\overset{d_0}{\rightarrow}\mathcal{L}_0\rightarrow\mathcal{E}\rightarrow 0$, where $c=\text{codim}(X, \P_k^N)$ and $\mathcal{L}_i=\mathcal{O}_X(-i)^{b_i}$ \cite{beu},  hence $\mathcal{E}^{\vee}\simeq\mathcal{L}^{\bullet}$ in D$^b_{\text{coh}}(X)$, where $\mathcal{L}^{\bullet}$ is represented by $\ldots\rightarrow\mathcal{O}_X(i+1)^{b_{i+1}}\rightarrow \mathcal{O}_X(i)^{b_i}\rightarrow\mathcal{O}_X(i-1)^{b_{i-1}}\rightarrow\ldots$. }
\end{exm}

\begin{exm}{\em The Beilinson resolution of the diagonal shows that $\mathcal{O}, \mathcal{O}(-1),\ldots ,\mathcal{O}(-n)$ is a full exceptional sequence for $\text{D}^b_{\text{coh}}(\P_k^n)$ \cite[8.29]{huy}, so that $\mathcal{E}^{\bullet}$ is derived Ulrich relatively to $\mathcal{O}_{\P_k^n}(1)$ if and only if it belongs to the subcategory $\langle \mathcal{O}_{\P_k^n}\rangle$, i.e. any derived Ulrich object on $(\P_k^n, \mathcal{O}_{\P_k^n}(1))$ is obtained by taking finite direct sums or iterated cones with extremals shifts of the structure sheaf. 

It is thus more interesting to study derived Ulrich sheaves on the projective space relatively to non--tautological embeddings of it. 

More generally, one can extend the Beilinson spectral sequence to the case of a projective bundle $\pi:Y=\P\mathcal{E}\rightarrow X$, and obtain a semi-orthogonal decomposition 
\begin{align*}
\text{D}^b_{\text{coh}}(Y)=\langle \mathcal{O}_Y(-n),\ldots ,\mathcal{O}_Y(-1), \pi^*\text{D}^b_{\text{coh}}(X)\rangle,
\end{align*}
where $n+1=\text{rk}(\mathcal{E})$. Thus, the derived Ulrich sheaves on $(Y, \mathcal{O}_Y(1))$ are determined by those on $X$. }
\end{exm}

\begin{exm} {\em When $X$ is a nonsingular projective curve,  the derived Ulrich sheaves on $X$ are direct sums of shifted Ulrich bundles (by Corollary \ref{curves} and \cite[1.25]{br}). }
\end{exm}

\begin{exm}{\em Here we try to make justice to the first-look arridity surrounding Definition \ref{derul}. Assume we have a glueing data given in the language of a six-functor formalism: $Rj_*, j_!:\text{D}(U)\rightarrow\text{D}(X), j^*:\text{D}(X)\rightarrow\text{D}(U), i_*:\text{D}(Y)\rightarrow\text{D}(X), i^*, i^!:\text{D}(X)\rightarrow\text{D}(Y)$, satisfying the defining compatibilities \cite[IV,  4]{gm}. Then the Beilinson-Bernstein-Deligne theorem says that the sequence 
\begin{align*}
\text{D}(Y)\overset{i_*}{\rightarrow}\text{D}(X)\overset{j^*}{\rightarrow}\text{D}(U)
\end{align*} 
is an exact triple and that any pair of $t$ - structures $(\text{D}^{\leq 0}(Y), \text{D}^{\geq 0}(Y))$ and $(\text{D}^{\leq 0}(U), \text{D}^{\geq 0}(U))$ determines a unique compatible $t$ - structure on $\text{D}(X)$ (i.e., $i_*$ and $j^*$ are $t$-exact functor). Moreover, the $t$-structure given by this glueing process is explicit:
\begin{align*}
&\text{D}^{\leq 0}(X)=\{\mathcal{F}\in\text{D}(X); j^*\mathcal{F}\in\text{D}^{\leq 0}(U), i^*\mathcal{F}\in\text{D}^{\leq 0}(Y)\} \\
&\text{D}^{\geq 0}(X)=\{\mathcal{F}\in\text{D}(X); j^*\mathcal{F}\in\text{D}^{\geq 0}(U), i^!\mathcal{F}\in\text{D}^{\geq 0}(Y)\}.
\end{align*} 
The converse is more easy: if $\mathcal{C}\rightarrow\mathcal{D}\rightarrow\mathcal{E}$ is an exact triple of triangulated categories, then plainly any $t$-structure on $\mathcal{D}$ determines unique compatible $t$-structures on $\mathcal{C}$ and $\mathcal{E}$.

Thus,  Definition \ref{derul} gives the possibility to construct $non-standard$ Ulrich objects  by glueing translated $t$ - structures from D$(Y)$ and D$(U)$.  }
\end{exm}

\begin{rema}\label{switch} {\em When needed, we may switch the variables via duals in the definition of a derived Ulrich sheaf. Indeed, for any $\mathcal{E}^{\bullet}, \mathcal{F}^{\bullet}\in\text{D}_{\text{coh}}^b(X)$ we have, for each $i$, Ext$^i(\mathcal{E}^{\bullet}, \mathcal{F}^{\bullet})\simeq$ Ext$^i({\mathcal{F}^{\bullet}}^{\vee}, {\mathcal{E}^{\bullet}}^{\vee})$:
\begin{align*}
&R\mathcal{H}om({\mathcal{F}^{\bullet}}^{\vee}, {\mathcal{E}^{\bullet}}^{\vee}[i])=R\mathcal{H}om(R\mathcal{H}om(\mathcal{F}^{\bullet}, \O_X), {\mathcal{E}^{\bullet}}^{\vee}[i]) \\
&\simeq R\mathcal{H}om(\O_X, \mathcal{F}^{\bullet}\otimes {\mathcal{E}^{\bullet}}^{\vee}[i])\simeq R\mathcal{H}om(\O_X, \mathcal{F}^{\bullet}\otimes {\mathcal{E}^{\bullet}[-i]}^{\vee}) \\
&\simeq R\mathcal{H}om(\O_X, {\mathcal{E}^{\bullet}[-i]}^{\vee}\otimes\mathcal{F}^{\bullet})\simeq R\mathcal{H}om(\O_X, R\mathcal{H}om(\mathcal{E}^{\bullet}[-i], \mathcal{F}^{\bullet})) \\
&\simeq R\mathcal{H}om(\mathcal{E}^{\bullet}[-i], \mathcal{F}^{\bullet}),
 \end{align*}\\
so by applying $R\Gamma$ and the translation $[i]$ we get the conclusion \cite[3.3, Compatibilities]{huy}.}
\end{rema}

Next, we see that some basic properties of Ulrich vector bundles extends to derived Ulrich sheaves (cf. \cite[2]{beu}). 

\begin{prop}\label{divisor}Let $X$ be a nonsingular projective variety over $k$ and let $\mathcal{O}_X(1)$ be very ample on $X$ giving $X\hookrightarrow \mathbb{P}_k^N$. Say $\iota:Y\hookrightarrow X$ is a hyperplane section in this embedding, and let $\iota^*:\text{D}_{coh}^b(X)\rightarrow\text{D}_{coh}^b(Y)$ be the induced derived pullback functor. Then for any derived Ulrich sheaf $\mathcal{E}^{\bullet}$ on $X$ with respect to $(\mathcal{O}_X(i))_i$ we have that $\iota^*\mathcal{E}^{\bullet}$ is a derived Ulrich sheaf on $Y$ with respect to $(\iota^*\mathcal{O}_X(i))_i$
\end{prop}

\begin{proof} Clearly, the sequence of pullbacks is an ample one (Example \ref{ample}). Let $L_j=\mathcal{O}_X(j)$; we have
\begin{align*}
&\text{Ext}^i_{\text{D}^b_{\text{coh}}(Y)}(\iota^*\mathcal{E}^{\bullet}, (\iota^*L_j)^{\vee})\simeq\text{Ext}^i_{\text{D}^b_{\text{coh}}(Y)}(\iota^*\mathcal{E}^{\bullet}, \iota^*L_j^{\vee})\\
&(\text{because the pullback commutes with tensor products}\\ 
&\text{and commutes with duals thanks to the trace map})\\
&\simeq\text{Hom}_{\text{D}_{\text{coh}}^b(Y)}(\iota^*\mathcal{E}^{\bullet}, (\iota^*L_j^{\vee})[i])\simeq\text{Hom}_{\text{D}_{\text{coh}}^b(X)}(\mathcal{E}^{\bullet}, \iota_*(\mathcal{O}_Y\otimes \iota^*(L_j^{\vee}[i])))\\
&(\text{by the adjunction between pullback and direct image, and the pullback}\\
&\text{commutes with translation since it is applied on a locally free sheaf})\\
&\simeq\text{Hom}_{\text{D}_{\text{coh}}^b(X)}(\mathcal{E}^{\bullet}, \iota_*\mathcal{O}_Y\otimes L_j^{\vee}[i])\\
&(\text{by the projection formula})\\
&\simeq\text{Ext}^i_{\text{D}^b_{\text{coh}}(X)}(\mathcal{E}^{\bullet}, \iota_*\mathcal{O}_Y\otimes L_j^{\vee})
\end{align*}

Now let $d=$ dim$X$ and $1\leq j\leq d-1$. Because $Y$ is a hyperplane section of $X$ by means of $L$, we have the short exact sequence (in $\textbf{Coh}(X))$ $0\rightarrow L_{j+1}^{\vee}\rightarrow L_j^{\vee}\rightarrow \iota_*\O_Y\otimes L_j^{\vee}\rightarrow 0$. We see this sequence in D$_{\text{coh}}^b(X)$ according to Remark \ref{emb}. Then, since any short exact sequence in the category of bounded complexes Kom$^b$ is quasi-isomorphic with a distinguished triangle \cite[III.3.5]{gm}, we have a long exact sequence
\begin{align*}
&\dots\rightarrow\text{Ext}^i(\mathcal{E}^{\bullet}, L_{j+1}^{\vee})=0\rightarrow\text{Ext}^i(\mathcal{E}^{\bullet}, L_j^{\vee})=0\rightarrow \\
&\rightarrow\text{Ext}^i(\mathcal{E}^{\bullet}, \iota_*\mathcal{O}_Y\otimes L_j^{\vee})\rightarrow\text{Ext}^{i+1}(\mathcal{E}^{\bullet}, L_{j+1}^{\vee})=0\rightarrow \dots
\end{align*}
according to the hypothesis also. Therefore, for any $i$ and $1\leq j\leq$ dim$Y$ we have Ext$^i_{\text{D}_{\text{coh}}^b(Y)}(\iota^*\mathcal{E}^{\bullet}, (\iota^*L_j)^{\vee})=0$, hence the conclusion. 
\end{proof}

Now the other way around:

\begin{prop}Let $\pi:X\rightarrow Y$ be a finite morphism of nonsingular projective varieties over $k$, and let $\mathcal{O}_X(1)=\pi^*\mathcal{O}_Y(1)$, where $\mathcal{O}_Y(1)$ is (say) some ample line bundle on $Y$. Given $\mathcal{E}^{\bullet}\in\text{D}^b_{\text{coh}}(X)$, then ${\mathcal{E}^{\bullet}}^{\vee}$ is a derived Ulrich sheaf for $(X, \mathcal{O}_X(1))$ if and only if $(\pi_*\mathcal{E}^{\bullet})^{\vee}$ is a derived Ulrich bundle for $(Y, \mathcal{O}_Y(1))$.
\end{prop}

\begin{proof} Indeed,
\begin{align*}
&\text{Ext}^i((\pi_*\mathcal{E}^{\bullet})^{\vee}, \mathcal{O}_Y(j))\overset{\text{Remark \ref{switch}}}{\simeq}\text{Ext}^i_{\text{D}_{\text{coh}}^b(Y)}(\mathcal{O}_Y(-j), \pi_*\mathcal{E}^{\bullet}) \\
&\simeq\text{Hom}_{\text{D}_{\text{coh}}^b(Y)}(\mathcal{O}_Y(-j), (\pi_*\mathcal{E}^{\bullet})[i]) \\
&\overset{\pi_*\circ\text{[ ]}\simeq\text{[ ]}\circ \pi_*}{\simeq}\text{Hom}_{\text{D}_{\text{coh}}^b(Y)}(\mathcal{O}_Y(-j), \pi_*(\mathcal{E}^{\bullet}[i])) \\
&\overset{\text{by adjunction}}{\simeq}\text{Hom}_{\text{D}^b_{\text{coh}}(X)}(\mathcal{O}_X(-j), \mathcal{E}^{\bullet}[i]) \\
&\simeq\text{Ext}^i({\mathcal{E}^{\bullet}}^{\vee}, \mathcal{O}_X(j))
\end{align*}
where the third identification is justified by the fact that $\pi$ is finite, and hence $R\pi_*=\pi_*$. 
\end{proof}

Also, the Newton binomial check: 

\begin{prop}Let $X$ and $Y$ be a nonsingular projective varieties over $k$, and $(L_i), (M_j)$ be ample sequences given by some projective embeddings of $X$ and $Y$, respectively. If $\mathcal{E}^{\bullet}\in$ ob$(D_{coh}^b(X))$ and $\mathcal{F}^{\bullet}\in$ ob$(D_{coh}^b(Y))$ are derived Ulrich sheaves with respect to $(L_i)$ and $(M_j)$ respectively, then the exterior (derived) tensor product $\mathcal{E}^{\bullet}\boxtimes \mathcal{F}^{\bullet}$ is a derived Ulrich sheaf on $X\times Y$ relatively to the sequence $(L_i\boxtimes M_{i-dim(X)})$. 
\end{prop}

\begin{proof} Let $d=\text{dim}X+\text{dim}Y\overset{denote}{=}d_1+d_2$. The exterior tensor product of very ample line bundles is a very ample line bundle (Segre) and by using the derived K\"{u}nneth formula \cite[3.33,(ii)]{huy} we get that, for any $j$ and any $1\leq i\leq d$, 
\begin{align*}
\text{Ext}^j(\mathcal{E}^{\bullet}\boxtimes \mathcal{F}^{\bullet}, L_i^{\vee}\boxtimes M_{i-d_1}^{\vee})\simeq \bigoplus\limits_{r+s=j}\text{Ext}^r(\mathcal{E}^{\bullet}, L_i^{\vee})\otimes\text{Ext}^s(\mathcal{F}^{\bullet}, M_{i-d_1}^{\vee})=0,
\end{align*}
hence the conclusion. 
\end{proof}

We shall need two lemmas.  They are certainly well--known,  but we couldn't find proofs for them,  so we supply them with proofs (besides,  the proof of the second lemma will be needed).

\begin{lema}\label{l1}\cite[1.37]{huy} If $A\overset{u}{\rightarrow} B\overset{v}{\rightarrow} C\overset{t}{\rightarrow} A[1]$ is a distinguished triangle in a triangulated category such that $t$ is trivial, then the triangle is split, i.e.  is given by a direct sum decomposition $B\simeq A\oplus C$. 
\end{lema}

\begin{proof} Because the category is additive, we may form $A\oplus C$. Then complete $i_1:A\rightarrow A\oplus C$ to a distinguished triangle $A\rightarrow A\oplus C\rightarrow C_0\rightarrow A[1]$ \cite[IV.1.1.TR1.(c)]{gm}. By applying \cite[IV.1.1.TR2]{gm} twice to the previous sequence and the one from hypothesis we get that $C\overset{t}{\rightarrow} A[1]\overset{-u[1]}{\rightarrow} B[1]\overset{-v[1]}{\rightarrow} C[1]$ and $C_0\rightarrow A[1]\rightarrow A[1]\oplus C[1]\rightarrow C_0[1]$ are distinguished triangles (the translation functor is additive).   Thus, as $t=0$, \cite[IV.1.2.5]{gm} gives a commutative diagram as follows:
\[
\begin{tikzcd}
C \arrow[d, dashrightarrow, "\gamma"] \arrow{r}  & A[1] \arrow[d, "id"] \arrow{r} & B[1] \arrow[d, dashrightarrow, "\beta"] \arrow{r} & C[1] \arrow[d, dashrightarrow, "\gamma{[1]}"]\\
C_0 \arrow{r} & A[1] \arrow{r} & A[1]\oplus C[1] \arrow{r} & C_0[1] \\
\end{tikzcd}
\]
Shifting backwards, we get the following commutative diagram:
\[
\begin{tikzcd}
A \arrow[d, "id"] \arrow{r}{u}  & B \arrow[d, "\beta{[-1]}"] \arrow{r}{v} & C \arrow[d, "\gamma"] \arrow{r}{t} & A[1] \arrow[d, "id"]\\
A \arrow{r}{i_1} & A\oplus C \arrow{r} & C_0 \arrow{r} & A[1] \\
\end{tikzcd}
\]
Then let $u':=p_1\circ\beta[-1]\in\text{Hom}(B, A)$, where $p_1$ is the first projection $A\oplus C\rightarrow A$. Indeed, we have $u'\circ u=p_1\circ (\beta[-1]\circ u)=p_1\circ i_1=\text{id}_A$. 

Now let $D$ be an arbitrary object in the category. We then have the long exact sequence arising from Hom$(- ,D)$ and the triangle in the hypothesis, from which we extract a short exact sequence in the category of abelian groups $\textbf{Ab}$:
\begin{align*}
0\rightarrow\text{Hom}(C, D)\rightarrow\text{Hom}(B, D)\rightarrow\text{Hom}(A, D)\rightarrow 0
\end{align*}
But the previous diagram says that this sequence splits in $\textbf{Ab}$, so the splitting lemma gives
\begin{align*}
\text{Hom}(B, D)\simeq\text{Hom}(A, D)\oplus\text{Hom}(C, D)\simeq\text{Hom}(A\oplus C, D).
\end{align*}

Therefore, since $D$ was arbitrary, Yoneda lemma gives $B\simeq A\oplus C$. 
\end{proof}

\begin{lema}\label{l2} \cite[2.33]{huy} Say $\mathcal{A}$ is an abelian category, and let $A^{\bullet}\in\text{ob(D}^+(\mathcal{A}))$. If $m=\text{min}\{i; \text{H}^i(A^{\bullet})\neq 0\}$, then there exists a distinguished triangle 
\begin{align*}
\text{H}^m(A^{\bullet})[-m]\rightarrow A^{\bullet}\overset{\varphi}{\rightarrow} B^{\bullet}\rightarrow\text{H}^m(A^{\bullet})[1-m]
\end{align*}
in D$^+(\mathcal{A})$ such that H$^i(B^{\bullet})=0$ for all $i\leq m$ and $\varphi$ inducing isomorphisms H$^i(\varphi):\text{H}^i(A^{\bullet})\simeq\text{H}^i(B^{\bullet})$ for each $i>m$.
\end{lema}

\begin{proof} Since H$^i(A^{\bullet})=0$ for all $i<m$, the commutative diagram
\[
\begin{tikzcd}
... \arrow{r} & A^{m-1} \arrow{d} \arrow{r} & A^m \arrow{d} \arrow{r} & A^{m+1} \arrow[d, "id"] \arrow{r} & ... \\
... \arrow{r} & 0 \arrow{r} & \text{Coker}(d_{A^{\bullet}}^{m-1}) \arrow{r} & A^{m+1} \arrow{r} & ... \\
\end{tikzcd}
\]
establish an isomorphism in D$^{+}(\mathcal{A})$ \cite[III.2.1.(a)]{gm}, showing that we may suppose $A^{\bullet}$ is such that $A^i=0$ for all $i<m$. 

As such, H$^m(A^{\bullet})=\text{Ker}(d^m_{A^{\bullet}})$, so we have the canonical injection H$^m(A^{\bullet})\hookrightarrow A^m$. But Remark \ref{emb} says that, in D$^+(\mathcal{A})$, H$^m(A^{\bullet})$ is a complex concentrated in degree $0$, so that in order to get a morphism of complexes we have to translate by $[-m]$. Thus, we have a morphism in Kom$^+(\mathcal{A})$, H$^m(A^{\bullet})[-m]\rightarrow A^{\bullet}$, inducing the identity on the $m$-th cohomology, and hence one such in D$^+(\mathcal{A})$ (by the afformentioned result from \cite{gm}). By completing this morphism to a distinguished triangle and then using the associated long exact sequence in cohomology we get the conclusion. 
\end{proof}

\begin{prop}\label{dh1} Let $\mathcal{A}$ be an abelian category with dh$(\mathcal{A})=1$. Then for any $X^{\bullet}\in\text{ob(D}^b(\mathcal{A}))$ we have $X^{\bullet}\simeq\bigoplus_iA_i[i]$, for some $A_i\in\text{ob}(\mathcal{A})$. 
\end{prop}

\begin{proof} All the necessary tools are available in order to proceed inductively on the length of $X^{\bullet}$: with Lemma \ref{l1} and Lemma \ref{l2} established, the proof goes verbatim with the proof of \cite[3.15]{huy}. 
\end{proof}

\begin{coro}\label{curves} Let $\mathcal{A}$ be an abelian category with $\text{dh}(\mathcal{A})=1$ and $(L_i)_i$ a sequence in ob$(\mathcal{A})$. If $X^{\bullet}\in\text{D}^b(\mathcal{A})$ is an Ulrich object with respect to $(L_i)$, then there is an Ulrich object in $\mathcal{A}$. 
\end{coro}

\begin{proof} According to Proposition \ref{dh1}, we have $X^{\bullet}\simeq \bigoplus_i A_i[i]$, for some $A_i\in\text{ob}(\mathcal{A})$. As we are in D$^b(\mathcal{A})$, this direct sum is finite; moreover, the Ext's commutes with finite direct sums in the first variable. Then we get that, for any $j$, $0=\text{Ext}_{D^b(\mathcal{A})}^j(X^{\bullet}, L^{\vee})\simeq\bigoplus_i\text{Ext}^{j-i}(A_i, L_1^{\vee})$. Thus, Ext$^j(A_i, L_1^{\vee})=0$ for any $j$, so each $A_i$ is an Ulrich object in $\mathcal{A}$. 
\end{proof}

Moving on: 

\begin{prop}\label{dh2} Let $\mathcal{A}$ be an abelian category with $\text{dh}(\mathcal{A})=2$ and $(L_i)_i$ be a sequence in ob$(\mathcal{A})$. If $A^{\bullet}\in\text{D}^b(\mathcal{A})$ is an Ulrich object with respect to $(L_i)$, then there is an Ulrich object in $\mathcal{A}$ relatively to the same sequence.
\end{prop}

\begin{proof} In dimension one, the process of direct sum was enough (Proposition \ref{dh1}); in general, we have a Grothendieck spectral sequence as follows:
\begin{align*}
\text{E}_2^{p,q}=\text{Ext}^p(\text{H}^{-q}(A^{\bullet}), L_i)\Rightarrow\text{Ext}^{p+q}(A^{\bullet}, L_i)
\end{align*}
\cite[2.70]{huy}. So the second page consists of just three columns, the lines $x=0, x=1, x=2$, since $\text{dh}(\mathcal{A})=2$. Because around the terms E$_2^{1,q}$ the differentials comes in from a zero term and comes out into a zero term, the line $x=1$ stabilizes, hence E$_2^{1,q}=$E$_{\infty}^{1,q}=0$, for any $q$. We now look at the third page. Here, $d_3^{0, q}:\text{E}_3^{0, q}\rightarrow\text{E}_3^{3, q-2}=\text{H(E}_2^{5, q-3})=0$ and $d_3^{-3, q+2}:\text{E}_3^{-3, q+2}=0\rightarrow\text{E}_3^{0, q}$, and this remains constant in position $(0, q)$, so E$_3^{0, q}=\text{E}_{\infty}^{0, q}=0$ for all $q$. But E$_3^{0, q}=\text{H(E}_2^{2, q-1})$, hence Ker$(d_2^{2, q-1})=\text{Im}(d_2^{0, q})$; as $d_2^{2, q-1}:\text{E}_2^{2, q-1}\rightarrow\text{E}_2^{4, q-2}=0$, we conclude that $d_2^{0, q}:\text{E}_2^{0, q}\rightarrow\text{E}_2^{2, q-1}$ is surjective. Next, we have $d_2^{-2, q+1}:\text{E}_2^{-2, q+1}=0\rightarrow\text{E}_2^{0, q}$ and $\text{E}_3^{-2, q+1}=\text{H(E}_2^{0, q})=\text{Ker}(d_2^{0, q})/\text{Im}(d_2^{-2, q+1})=\text{Ker}(d_2^{0, q})$. As before, we find E$_3^{-2, q+1}=\text{E}_{\infty}^{-2, q+1}=0$ for all $q$. Therefore, the spectral sequence gives 
\begin{align*}
\text{Ext}^1(\text{H}^{-q}(A^{\bullet}), L_i)=0\text{ and  Ext}^0(\text{H}^{-q}(A^{\bullet}), L_i)\overset{d_2}{\simeq}\text{Ext}^2(\text{H}^{-q-1}(A^{\bullet}), L_i), (*)
\end{align*}
for any $i=-1, -2$ and any $q$.

Next, since $A^{\bullet}$ comes from a bounded complex and it is not trivial, it makes sense to consider $M=\text{max}\{q; \text{H}^q(A^{\bullet})\not\simeq 0\}$, and let $\mathcal{H}=\text{H}^M(A^{\bullet})$. Thus, taking $q=-M$ and, respectively, $q=-M-1$ in $(*)$ gives $\text{Ext}^1({\mathcal{H}}, L_i)=0=\text{Ext}^2({\mathcal{H}}, L_i)$. Again, by the definition of $M$, there is clearly a morphism $A^{\bullet}\rightarrow\mathcal{H}[-M]$, inducing the identity on the $M$-th cohomology (according to the proof of Lemma \ref{l2}), and hence a morphism $A^{\bullet}[M]\rightarrow\mathcal{H}$ inducing the identity on the $0$-th cohomology. So, if $\text{Hom}(\mathcal{H}, L_i)\neq 0$ is to be the case, then by composing such a non-zero morphism with the previous one we get a non-zero morphism $A^{\bullet}[M]\rightarrow L_i$, and hence the contradiction $\text{Ext}^{-M}(A^{\bullet}, L_i)\simeq\text{Hom}(A^{\bullet}, L_i[-M])\neq 0$. Indeed, the composition can not be trivial because otherwise it will follow that, in particular, it induces the zero morphism on the $0$-th cohomology, which is absurd because this induced morphism has the form $(non-zero)\circ\text{id}$ (also by Remark \ref{emb} !).

Therefore, $\text{Ext}^j(\mathcal{H}, L_i)=0$ for all $j$ and each $i=-1, -2$ (or $i=1, 2$), i.e.  $\mathcal{H}$ is an Ulrich object in $\mathcal{A}$ with respect to $(L_i)_i$. 
\end{proof}

We now observe a way of manufacturing Ulrich bundles on surfaces:

\begin{coro}\label{manuf}Let $(X, \mathcal{O}_X(1))$ be a nonsingular projective variety over $k$ of dimension $d$, and suppose $\mathcal{E}^{\bullet}$ is a non-trivial derived Ulrich sheaf on $X$ relative to $\mathcal{O}_X(1)$.  Let $\mathcal{H}=\text{H}^M(\mathcal{E}^{\bullet})$, where $M=\text{max}\{i; \text{H}^i(\mathcal{E}^{\bullet})\neq 0\}$ If $d=3$ and $Z$ is a smooth hyperplane section of $X$, then $\mathcal{H}_{|Z}$ is Ulrich.  If $d=2$, then $\mathcal{H}^{\vee}$ is an Ulrich bundle on $X$.
\end{coro}

\begin{proof} There is a spectral sequence
\begin{align*}
\text{E}^{p, q}_2=\text{L}^p\iota^*(\text{H}^q(\mathcal{E}^{\bullet}))\Rightarrow\text{L}^{p+q}\iota^*(\mathcal{E}^{\bullet})
\end{align*}  
\cite[page 81, 3.10)]{huy}. As usual, since $\iota$ is the inclusion of a divisor, we have degeneration at the second page which gives, for each $i\in\Z$, short exact sequences as follows:
\begin{align*}
0\rightarrow\text{L}^0\iota^*\text{H}^i(\mathcal{E}^{\bullet})\rightarrow\text{H}^i(\iota^*\mathcal{E}^{\bullet})\rightarrow\text{L}^1\iota^*\text{H}^{i+1}(\mathcal{E}^{\bullet})\rightarrow 0
\end{align*}
Thus, taking $i=M$ gives $\text{L}^0\iota^*\text{H}^M(\mathcal{E}^{\bullet})\simeq\text{H}^M(\iota^*\mathcal{E}^{\bullet})$. But $L^0F$ is $F$ on $\textbf{Coh}(X)$, so the left part of this isomorphism is just the usual restriction $\iota^*(\text{H}^M(\mathcal{E}^{\bullet}))$. Thus, the second assertion (and also the third one) follows from Proposition \ref{divisor} and Proposition \ref{dh2} (also by Remark \ref{switch}).
\end{proof}

\section{Acknowledgments} I wish to present my thanks to professor Victor Vuletescu for reading the first version and improve on it, and to professor Marian Aprodu for introducing to me the notion of Ulrich bundle.

\end{document}